\documentclass[11pt,a4paper, reqno]{amsart}
\usepackage{amsmath,amsfonts,amssymb,mathrsfs}
\usepackage{mathtools}
\usepackage[left=3cm,right=3cm,top=3cm,bottom=3cm,bindingoffset=0cm]{geometry}
\usepackage{amsthm}
\usepackage{amsbsy}
\usepackage{longtable}
\usepackage{esint}
\usepackage{graphics}
\usepackage[dvips]{graphicx}
\usepackage{psfrag}
\usepackage{epsfig}
\usepackage{subfig}
\usepackage{color}		
\usepackage{upgreek}
\usepackage{hyperref}

\numberwithin{equation}{section}

\newcommand{\R}{\mathbb{R}}
\newcommand{\N}{\mathbb{N}}

\newtheorem{teo}{Theorem}[section]

\newtheorem{rem}[teo]{Remark}

\newtheorem{lemma}[teo]{Lemma}

\renewcommand{\div}{\operatorname{div}}

\begin{document}

\author{Nunzia Gavitone}
\address[N.\,Gavitone]{Dipartimento di Matematica e Applicazioni ``Renato Caccioppoli'',
Universit\`a degli Studi di Napoli Federico II, Via Cintia, Monte S.\ Angelo,
80126 Napoli, Italy}
\email{nunzia.gavitone@unina.it}
\author{Riccardo Molinarolo}
\address[R.\,Molinarolo]{Dipartimento di Informatica, Universit\`a degli Studi di Verona, Strada le Grazie 15, 37134 Verona, Italy}
\email{riccardo.molinarolo@univr.it}
\thanks{\today}
\title[On the Serrin's problem with Robin boundary conditions]{On the Serrin's problem with Robin boundary conditions}

\begin{abstract}
Let $\Omega \subset \R^N$, $N\ge 2$, be an open, connected, bounded set with $C^2$ boundary. In this paper we consider the torsion problem with Robin boundary conditions and we study the symmetry of the solutions when  suitable extra conditions are imposed on the boundary of $\Omega$. In particular, we prove the Serrin's rigidity result under suitable assumptions on the domain and on the Robin parameter.
\end{abstract}

\maketitle

\par

\bigskip

\bigskip

\section{Introduction}
Let $\Omega \subset \R^N$, $N\ge 2$, be an open, connected, bounded subset with $C^2$ boundary and $\beta \in \mathbb \R\setminus \{0\}$. Let us consider the torsion problem with Robin boundary conditions, that is
\begin{equation}\label{PbR}
    \begin{cases}
    \Delta u = N &\text{in } \Omega,
    \\
    \dfrac{\partial u}{\partial \nu}+ \beta u = 0 &\text{on } \partial\Omega,
    \end{cases}
\end{equation}
where $\nu$ denotes the outer unit normal to the boundary. 
The existence of a unique solution to the problem \eqref{PbR} depends on the Robin parameter $\beta$. It is well-known, in fact,  that  if $\beta>0$ there exists a unique  solution $u$ to the problem \eqref{PbR} and it is smooth and negative.
In the case $\beta <0$, in \cite{bw} and \cite{bw_neg} the authors prove that the problem \eqref{PbR} has a unique solution $u$ if  $|\beta| \neq \mu_i(\Omega)$ for every $i \in \N_0$, where  
\[
0 = \mu_0(\Omega) < \mu_1(\Omega) \leq \mu_2(\Omega) \leq \, ... \, \leq \mu_i(\Omega) \leq \, ... \, \nearrow +\infty, 
\]
are the Steklov eigenvalues of the Laplacian. Let us observe that, for $\beta < 0$, the function $u$ could change sign (see Section \ref{sec rad tor}). 
In what follows, in order to have the existence of a unique solution $u$  to the problem \eqref{PbR}, we will always assume the following condition on the Robin parameter $\beta$:
\begin{equation}\label{beta condition}
    \beta \neq - \mu_i(\Omega) \quad \forall i \in \N_0,
\end{equation}
which is trivially verified when $\beta>0$.

Then we can consider the following quantity
\[
T_{\beta}(\Omega)=\frac{1}{N}\int_{\Omega}(-u)\,dx,  
\]
which is the so-called Robin torsional rigidity of $\Omega$. In recent years, the study of shape optimization problems related to $T_{\beta}(\Omega)$ under geometrical constraints has been the subject of extensive research by many mathematicians. In this order of idea, we remind the readers that the behavior of  $T_{\beta}(\Omega)$ under volume constraint and in particular the existence of optimal sets which minimizes or maximizes $T_{\beta}(\Omega)$ among the domains with fixed measure, depends on the Robin parameter $\beta$. In the case $\beta>0$, in \cite{bg} the authors showed the following result:
\begin{equation}
\label{bg pos}
T_{\beta}(\Omega)\le T_{\beta}(\Omega^{\sharp}),
\end{equation}
where $\Omega^{\sharp}$ is the ball having the same measure than $\Omega$.
The inequality \eqref{bg pos} asserts that, among all domains of given volume, the ball has the largest torsional rigidity and this generalizes the classical well-known result when Dirichlet boundary conditions are imposed (see \cite{polya,talenti}). 

In the case $\beta<0$, in \cite{bw_neg} the authors proved that, in the plane, if $|\beta|<\mu_1(\Omega)$, then
\begin{equation*}
T_{\beta}(\Omega^{\sharp})\le T_{\beta}(\Omega).
\end{equation*}

Then for $\beta<0$, in the plane the disk has the smallest torsional rigidity under volume constraint if $|\beta|< \mu_1(\Omega)$, while in higher dimension this result, up to our knowledge, is still open.  However, we recall that in \cite{bw}, in any dimension, the authors proved that for $|\beta|$ sufficiently small, the ball is a local maximum by using a shape derivative approach. To do that, the authors computed the first domain derivative of the torsion functional and, for any $\beta \in \R\setminus\{0\}$ such that \eqref{beta condition} holds, they proved that $\Omega$ is a critical domain under volume constraint if and only if it verifies the following extra condition on the boundary: 
\begin{equation}
\label{spr intro}
    |\nabla u|^2 +2Nu - 2\beta^2 u^2 + (N-1) \beta u^2 \mathcal{M}= C \quad \text{on } \partial\Omega,
\end{equation}
where $\mathcal{M}$ is the mean curvature of $\partial \Omega$ and $C$ is a constant.
Then, as observed in \cite{bw}, to characterize the critical domains of the functional $T_{\beta}$ under volume constraint, one could investigate the following overdetermined boundary value problem 
\begin{equation}\label{SpR intro}
        \begin{cases}
        \Delta u = N &\text{in } \Omega,
        \\
        \displaystyle \frac{\partial u}{\partial \nu}+ \beta u = 0 &\text{on } \partial\Omega,
        \\
        |\nabla u|^2 + 2Nu - 2\beta^2 u^2 + (N-1) \beta u^2 \mathcal{M} = C &\text{on } \partial\Omega,
      \end{cases}
\end{equation}
where $\nu $ denotes the outer unit normal to the boundary.
    
Taking into account the quoted results, this study should allow us to understand if there are critical domains for $T_{\beta}$ other than the ball. We observe that,  when $\beta$ goes to $+\infty$, that is in the Dirichlet case, \eqref{spr intro} reduces to (up to redefine the constant $C$)
\[
|\nabla u| = C \,\,\text{ on } \partial\Omega,
\]
and then \eqref{SpR intro} reduces to the celebrated Serrin's problem studied in \cite{serrin}. In that paper, Serrin proved the radial symmetry of  the solution, and then of the  domain,  by using the so-called moving plane technique.  We recall that in \cite{w},  Weinberger proved the same rigidity result by using the maximum principle for a suitable auxiliary function called P-function. Since these fundamental contributions, several alternative proofs and generalizations to linear and nonlinear operators  followed: without any claim of completeness, we mention for instance \cite{bnst,bgnt,bh,bk,cs,f,fgk,gnn,hp,gl,mmp,nt}. Finally, we recall that in the paper \cite{mp19} the authors obtained the rigidity result for the Serrin's problem by using a deep connection with the celebrated Alexandrov’s Soap Bubble Theorem. Among their results,  we recall that they also proved a rigidity result when the following extra boundary condition, involving the mean curvature of the domain, is imposed: 
\begin{equation}
\label{pogh intro}
\langle \nabla u, \nu \rangle \mathcal{M} = 1 \text{ on } \partial\Omega,
\end{equation}
where $\langle \cdot,\cdot \rangle$ denotes the standard Euclidean scalar product. 
We stress that in the Dirichlet case the following well-known Reilly formula  holds
\begin{equation*}
\Delta u= (N-1)\langle \nabla u,\nu\rangle \mathcal{M} + \langle (\nabla^2 u) \,\nu , \nu\rangle  \quad \text{ on } \partial \Omega,
\end{equation*}
and then \eqref{pogh intro} is equivalent to require the following alternative condition
\begin{equation}
\label{sec}
\displaystyle \langle \nabla^2 u \, \nu , \nu \rangle = 1 \quad \text{ on } \partial \Omega.
\end{equation}

The aim of this paper is to study some  overdetermined problems for the Robin torsion problem \eqref{PbR} in the spirit of \cite{mp19}.  Then we consider different kinds of extra conditions on the boundary as for instance  the natural Serrin-Robin overdetermined condition  \eqref{spr intro}, or conditions  of the type \eqref{pogh intro}  and \eqref{sec} that,  in the Robin setting, are not equivalent (see Section \ref{sec notation} for the general Reilly's identity). 
Our main results are contained in Section \ref{sec overdet problem}. Among them, 
as regards  problem \eqref{SpR intro}, denoting by $\kappa_i$ the $i^{th}$ principal curvature of $\partial \Omega$ ($\kappa_i$ are oriented so that convex sets have positive curvatures), and by 
\begin{equation}\label{eq kappa min intro}
    \kappa_{\min} = \min_{\partial \Omega} \min_{i=1,\dots,N-1}  \kappa_i(x),
\end{equation}
we obtain the following.
\begin{teo}[Rigidity for the Serrin's problem with Robin boundary conditions]
    Let $\Omega \subset \R^N$, $N\ge 2$, be an open, connected, bounded set with $C^2$ boundary. Let $\beta > 0$ and let $u \in C^2(\overline{\Omega})$ be a solution of
    \begin{equation*}
        \begin{cases}
        \Delta u = N &\text{in } \Omega,
        \\
        \displaystyle \frac{\partial u}{\partial \nu}+ \beta u = 0 &\text{on } \partial\Omega,
        \\
        |\nabla u|^2 + 2Nu - 2\beta^2 u^2 + (N-1) \beta u^2 \mathcal{M} = C &\text{on } \partial\Omega,
      \end{cases}
    \end{equation*}
    with 
    \begin{equation*}
        R = \frac{N |\Omega|}{|\partial\Omega|} \quad\text{and}\quad C = - R^2 - R \left(\frac{N+1}{\beta}\right).
    \end{equation*}
    Let $k_{\min}$ be as in \eqref{eq kappa min intro}. If
    \begin{equation}
    \label{sci}
    \beta + k_{\min} \ge 0,
    \end{equation}
    then $\Omega$ is a ball and $u$ is radially symmetric.
\end{teo}
We stress that, being condition \eqref{sci} clearly fulfilled in the Dirichlet case $\beta \to +\infty$, then our result generalizes the Serrin's one. The key ingredient of the proof is a fundamental integral identity in the spirit of  \cite{mp19} (see Section \ref{sec integral identities} for details). Finally, we also refer to  \cite{mmp} for further fundamental integral identities under constant Neumann boundary condition.
 
The structure of the paper is the following. In Section \ref{sec notation} we introduce some notation, we recall some known fact of differential geometry, and the radial solution for the torsional rigidity with Robin boundary conditions. In Section \ref{sec integral identities} we present some key integral identities and finally in the last Section \ref{sec overdet problem} we state and prove our rigidity results.
 
\section{Notation and preliminaries}
\label{sec notation}
\subsection{Some differential geometry's tools}
Let be $N \in \N $, with $N\ge 2$. We consider the Euclidean space $\R^N$, endowed with the standard Euclidean scalar product $\langle \cdot,\cdot \rangle$, and we denote by $B_R(z)$ the ball of center $z$ and of radius $R > 0$ (we set $B_R= B_R(0)$).

Let $\Omega \subset \R^N$ be an open, bounded set with $C^2$ boundary. Then, we denote by $\mathcal{M}$ the mean curvature of $\partial \Omega$, namely
\[
\mathcal{M}=\frac{1}{N-1} \sum_{i=1}^{N-1} \kappa_i,
\]
where $\kappa_i$ for $i=1, \dots,N-1$ denote the $i^{\text{th}}$ principal curvature of $\partial \Omega$ ($\kappa_i$ are oriented so that convex sets have positive curvatures).

In what follows $\nu$ denotes the unit outer normal to the boundary. For a function $f\in C^2(\overline{\Omega})$, we denote by $\nabla f = \left(\frac{\partial f}{\partial x_1},\dots,\frac{\partial f}{\partial x_N} \right)$ the gradient of $f$, where $\frac{\partial f}{\partial x_i}$ denotes the partial derivative with respect to the element of the base $x_i$ of $\R^N$. Moreover, we denote by $\nabla^\tau f$ the tangential gradient of the function $f$ along $\partial \Omega$, given by
\begin{equation}
\label{gradtan}
\nabla^{\tau}f= \nabla f- \dfrac{\partial f}{\partial \nu} \nu = \nabla f - \langle \nabla f, \nu \rangle \nu.
\end{equation}
Here and throughout the paper, we adopt the Einstein summation convention for
repeated indices.
Analogously, for any smooth vector field $v\colon \Omega \to \R^N$ the tangential divergence is given by
\begin{equation*}
\div_{\tau} v=  \div v- \langle (\nabla v) \,\nu , \nu\rangle, 
\end{equation*}
where $\nabla v \in \R^{N \times N}$. In the sequel we will frequently use the Divergence Theorem on $\partial \Omega$, hence we recall it (see for instance \cite{giusti,li,hepi}).

\begin{teo}[Divergence Theorem on $\partial \Omega$]\label{thm: divergence part Om}
    Let $\Omega$ be an open, bounded subset of $\R^N$ with $C^2$ boundary. Let $f\in C^1(\partial \Omega) $ and $v \in C^{1}(\partial \Omega, \R^N)$. Then the following integration by parts formula holds:
    \begin{equation*}
    \int_{\partial \Omega}f\div_{\tau}v\, d\mathcal H^{N-1}=-\int_{\partial \Omega} \langle v,\nabla^{\tau} f\rangle \, d\mathcal H^{N-1} + (N-1)\int_{\partial \Omega} f \langle v,\nu\rangle \mathcal{M} \, d\mathcal H^{N-1}.
    \end{equation*} 
    In particular, for $v \in C^{2}(\partial \Omega)$, we have
    \begin{equation*}
    \int_{\partial \Omega} f \Delta_{\tau} v \, d\mathcal H^{N-1}=-\int_{\partial \Omega} \langle \nabla^{\tau}v,\nabla^{\tau} f\rangle \, d\mathcal H^{N-1} .
    \end{equation*} 
    where we set $\Delta_{\tau} v := \div_{\tau} \nabla^{\tau} v$.
\end{teo}

We notice that, if $\Omega \subset \R^N$ is an open, bounded set with $C^2$ boundary, then
\begin{equation*}
\div_{\tau} \nu = (N-1) \mathcal{M},    
\end{equation*} 
and for any $v \in C^{1}(\overline{\Omega}, \R^N)$ one has
\begin{equation*}
\div v = \div_{\tau} v^{\tau} + (N-1)\langle v,\nu\rangle \mathcal{M} + \langle (\nabla v) \,\nu , \nu\rangle  \quad \text{on } \partial \Omega,
\end{equation*}
where $v^\tau$ is the tangential component of $v$. In particular, for a function $v \in C^2(\overline{\Omega})$, one has the well-known Reilly's identity:
\begin{equation}\label{eq decomposition Delta u}
\Delta v= \Delta_{\tau} v + (N-1)\langle \nabla v,\nu\rangle \mathcal{M} + \langle (\nabla^2 v) \,\nu , \nu\rangle  \quad \text{on } \partial \Omega
\end{equation}
(see \cite{re74, re77, re82} or \cite{hepi}).

\subsection{The radial torsional rigidity}
\label{sec rad tor}
In this section we briefly recall some known results  for the problem \eqref{PbR} in the radial case, that is when  $\Omega$ is a ball. Then we have the following problem
\begin{equation}\label{Trad}
    \begin{cases}
    \Delta q = N &\text{in } B_R(z),
    \\
    \dfrac{\partial q}{\partial \nu}+ \beta q = 0 &\text{on } \partial B_R(z),
    \end{cases}
\end{equation}
where we recall that $B_R(z)$ is the ball centered at $z$ with radius $R>0$, $\beta \in \mathbb{R} \setminus \{0\}$ and $\nu$ is the outer unit normal to the boundary. When $-\beta \neq \displaystyle \mu_k(B_R)=\frac{k}{R}$,  the solution $q$ to \eqref{Trad} can be explicitly computed and it is given by
\begin{equation}\label{eq q radial}
    q(x)= \frac{1}{2} |x-z|^2 -\frac{1}{2} R^2 - \frac{R}{\beta} \quad \text{for every } x \in \overline{B_R(z)}.
\end{equation}
We notice that, clearly, 
\[
\nabla q = x-z \quad \text{and}\quad \Delta q = N \quad \text{in } \overline{B_R(z)}, 
\]
while
\[
q = -\frac{R}{\beta} \quad \text{and} \quad \langle\nabla q, \nu \rangle= R \quad \text{on }\partial B_R(z).
\]
Moreover, in the radial setting, one can explicitly evaluate the constant $C$ of the overdetermined condition \eqref{spr intro}, namely the following holds:
\begin{equation*}
C=|\nabla q|^2 + 2Nq - 2\beta^2 q^2 + (N-1) \beta q^2 \mathcal{M}_0 = - R^2 - R \left(\frac{N+1}{\beta}\right)  \quad \text{on }\partial B_R(z),
\end{equation*}
where $\mathcal{M}_0 = \frac{1}{R}$ is the constant mean curvature of the ball $B_R(z)$. 

We mention that, at least for $\beta > 0$, the function $q$ and the constant $C$ are negative, and for $\beta \to +\infty$ the constant $C$ tends to the value $-R^2$, which is consistent with the classical Serrin problem with homogeneous Dirichlet boundary conditions. Instead, when $\beta < 0$, we observe that the function $q$ has not, in general, constant sign and the same holds for the constant $C$.

\section{Integral identities} \label{sec integral identities}
In this section we are going to compute different integral identities which will be used to prove our main rigidity results. In order to do that, we first provide the following lemma which will be useful in the sequel (see also \cite[Lemma 2.1]{mmp}).

\begin{lemma}\label{lemma identity}
    Let $\Omega$ be an open, connected, bounded subset of $\R^N$ with $C^2$ boundary. Let $\beta \in \mathbb{R}\setminus\{0\}$ with \eqref{beta condition} in force. Let $u\in C^2(\overline{\Omega})$ be a solution to the problem \eqref{PbR}. Then 
    \begin{equation}
    \label{fl}
    \langle \nabla^2 u \nabla u, \nu \rangle = \langle \nabla u, \nu \rangle \langle \nabla^2 u \, \nu , \nu \rangle - \sum_{i=1}^{N-1} \left(\beta + \kappa_i \right) \langle \nabla^\tau u, \theta_i \rangle^2 \quad \text{on } \partial\Omega,
    \end{equation}
    where $\{\kappa_i\}_{i=1,\dots,N-1}$ denote the principal curvatures of $\partial \Omega$ and $\{\theta_i\}_{i=1,\dots,N-1}$ is the associated tangential frame consisting of principal directions of $\partial\Omega$.
\end{lemma}

\begin{proof}
    Since $\Omega$ is of class $C^2$, we know that we can always extend the vector field $\nu$ smoothly to a tubular neighborhood $\mathcal{U}$ of $\partial \Omega$ (see, e.g., \cite{KP}). For instance, we can define $\nu = -\eta \nabla \mathrm{d}_{\partial\Omega}$, where the function $\mathrm{d}_{\partial\Omega}(x) := \mathrm{dist}(x,\partial\Omega)$ for every $x \in \overline{\Omega}$, and $\eta$ is a smooth cut-off function with support contained in the set in $\mathcal{U}$ and such that $\eta \equiv 1$ near $\partial\Omega$. Moreover, $\mathrm{d}_{\partial\Omega} \in C^2(\mathcal{U})$ and $|\nabla \mathrm{d}_{\partial\Omega}|^2 = 1$ in $\mathcal{U}$. Finally, by \cite[Appendix 14.6]{GT}, at any point in $\partial\Omega$, we can choose coordinates so that
    \begin{equation}\label{nabla nu}
        \nabla \nu = - \nabla^2 \mathrm{d}_{\partial\Omega} = 
        \left[\begin{matrix}
        \kappa_1  &\cdots &0 &0 \\
        \vdots &\ddots &\vdots  &\vdots \\
        0  &\cdots &\kappa_{N-1} &0 \\
        0  &\cdots &0 &0
        \end{matrix}
        \right].
    \end{equation}
    We now proceed assuming to extend the normal $\nu$ as described above, whenever it is required in the computation below. Since $u$ satisfies the Robin boundary conditions, then $\partial\Omega$ is a level surface for the function $\langle \nabla u, \nu \rangle + \beta u$, hence
    \begin{equation}\label{lemma eq 1}
        \nabla \Big( \langle \nabla u,\nu \rangle + \beta u \Big) = \langle \nabla \Big( \langle \nabla u,\nu \rangle + \beta u \Big), \nu \rangle \nu \quad \text{on } \partial\Omega.
    \end{equation}
    By direct computation we get
    \begin{equation*}
    \begin{aligned}
        \nabla \Big( \langle \nabla u,\nu \rangle + \beta u \Big) & = (\nabla^2 u) \nu + (\nabla \nu) \nabla u + \beta \nabla u 
        \\
        &=  
        (\nabla^2 u) \nu + (\nabla \nu) \Big( \nabla^\tau u +  \langle \nabla u,\nu \rangle \nu \Big)  + \beta \nabla u
        \\
        &= (\nabla^2 u) \nu + (\nabla \nu) \nabla^\tau u +  \langle \nabla u,\nu \rangle (\nabla \nu) \nu  + \beta \nabla u
        \\
        &= (\nabla^2 u) \nu + (\nabla \nu) \nabla^\tau u  + \beta \nabla u,
    \end{aligned}
    \end{equation*}
    where we have used the decomposition \eqref{gradtan} and the fact that $(\nabla \nu) \nu=0$ in $\mathcal{U}$, being the extension of $\nu$ unitary. Moreover, by \eqref{lemma eq 1}, we know that, for every $\theta_i$ with $i=1,\dots,N-1$, we have
    \begin{equation*}
        \langle \nabla \Big( \langle \nabla u,\nu \rangle + \beta u \Big), \theta_i \rangle = 0 \quad \text{on } \partial\Omega.
    \end{equation*}
    Hence, for every $i=1,\dots,N-1$, we obtain
    \begin{equation*}
        0 = \langle (\nabla^2 u) \nu + (\nabla \nu) \nabla^\tau u  + \beta \nabla u , \theta_i \rangle = \langle (\nabla^2 u) \nu , \theta_i \rangle + \langle (\nabla \nu) \nabla^\tau u , \theta_i \rangle + \langle \beta \nabla^\tau u , \theta_i \rangle \quad \text{on } \partial\Omega,
    \end{equation*}
    which, by the symmetry of $\nabla^2 u$, implies 
    \begin{equation*}
        \langle (\nabla^2 u) \theta_i , \nu \rangle = - \langle (\nabla \nu) \nabla^\tau u , \theta_i \rangle - \langle \beta \nabla^\tau u , \theta_i \rangle \quad \text{on } \partial\Omega.
    \end{equation*}
    Finally, again by \eqref{gradtan} and by the obvious decomposition $\displaystyle\nabla^\tau u = \sum_{i=1}^{N-1} \langle \nabla^\tau u, \theta_i \rangle \theta_i$, we have
    \begin{equation*}
    \begin{aligned}
        \langle (\nabla^2 u) \nabla u , \nu \rangle &=  \langle \nabla u,\nu \rangle \langle (\nabla^2 u) \nu , \nu \rangle + \langle (\nabla^2 u) \nabla^\tau u , \nu \rangle 
        \\ &=  \langle \nabla u,\nu \rangle \langle (\nabla^2 u) \nu , \nu \rangle + \sum_{i=1}^{n-1} \langle \nabla^\tau u, \theta_i \rangle \langle (\nabla^2 u) \theta_i , \nu \rangle
        \\ &=  \langle \nabla u,\nu \rangle \langle (\nabla^2 u) \nu , \nu \rangle - \sum_{i=1}^{n-1} \langle \nabla^\tau u, \theta_i \rangle 
        \Big( \langle (\nabla \nu) \nabla^\tau u , \theta_i \rangle + \langle \beta \nabla^\tau u , \theta_i \rangle  \Big) 
    \end{aligned}
    \end{equation*}
    on $\partial\Omega$. Using the fact that $\langle (\nabla \nu) \theta_i , \theta_i \rangle = \kappa_i$, we conclude.
\end{proof}

\begin{rem}\label{remark curvature}
We observe that by \eqref{nabla nu}, \eqref{fl}  can be also written in the following way
\begin{equation*}
  \langle \nabla^2 u \nabla u, \nu \rangle= \langle \nabla u, \nu \rangle \langle \nabla^2 u \, \nu , \nu \rangle - \langle (\nabla \nu ) \nabla^\tau u, \nabla^\tau u \rangle  - \beta |\nabla^\tau u|^2 \quad \text{on } \partial \Omega.
\end{equation*}
Moreover
    \begin{equation}
    \label{stima}
        \kappa_{\min} |\nabla^\tau u|^2 \leq \langle (\nabla \nu ) \nabla^\tau u, \nabla^\tau u \rangle \leq \kappa_{\max} |\nabla^\tau u|^2 \quad \text{on } \partial\Omega,
    \end{equation}
    where 
    \begin{equation}
    \label{k}
    \kappa_{\min} = \min_{\partial \Omega} \min_{i=1,\dots,N-1}  \kappa_i(x), \qquad \kappa_{\max} = \max_{\partial \Omega} \max_{i=1,\dots,N-1}  \kappa_i(x).
    \end{equation}
\end{rem}

For a solution $u$ of problem \eqref{PbR}, we set
\begin{equation}\label{eq P}
        P = \frac{1}{2} |\nabla u|^2 - u \quad\text{on }\overline{\Omega}.
\end{equation}
We are now in the position to prove our first identity. 

\begin{teo}[Fundamental Identity]\label{teo first identity}
    Let $\Omega$ be an open, connected, bounded subset of $\R^N$ with $C^2$ boundary. Let $\beta \in \mathbb{R} \setminus \{0\}$ with \eqref{beta condition} in force. Let $u\in C^2(\overline{\Omega})$ be a solution to the problem \eqref{PbR}. Let $P$ be as in \eqref{eq P}.
    Then the following identity holds:
    \begin{align*}
        \int_{\Omega} \Delta P \, dx = &- \int_{\partial\Omega} \langle (\nabla \nu ) \nabla^\tau u, \nabla^\tau u \rangle \,d\mathcal H^{N-1} - 2 \beta \int_{\partial\Omega} |\nabla^\tau u|^2 \,d\mathcal H^{N-1} 
        \\
        & - (N-1) \int_{\partial\Omega} \langle \nabla u, \nu \rangle \Big( \langle \nabla u, \nu \rangle \mathcal{M} - 1 \Big) \,d\mathcal H^{N-1}. 
    \end{align*}
\end{teo}

\begin{proof}
    By the Divergence Theorem we can write
    \begin{equation*}
        \int_{\Omega} \Delta P \, dx = \int_{\partial\Omega} \langle \nabla P, \nu \rangle  \,d\mathcal H^{N-1}.
    \end{equation*}
    Then to compute $\langle \nabla P, \nu \rangle$ we use the definition of $P$ and Lemma \ref{lemma identity}. We obtain
    \begin{align*}
        \langle \nabla P, \nu \rangle &= \langle \nabla^2 u \nabla u, \nu \rangle - \langle \nabla u , \nu \rangle 
        \\
        &= 
        \langle \nabla u, \nu \rangle \Big( N - \Delta_{\tau} u - (N-1)\langle \nabla u,\nu\rangle \mathcal{M} \Big) 
        \\
        &\quad 
        - \langle (\nabla \nu ) \nabla^\tau u, \nabla^\tau u \rangle  - \beta |\nabla^\tau u|^2 - \langle \nabla u , \nu \rangle \quad \text{on } \partial\Omega, 
    \end{align*}
    where in the last equality we have used Remark \ref{remark curvature} and the Reilly's identity \eqref{eq decomposition Delta u}.
    Thus,
    \begin{align*}
        \int_{\Omega} \Delta P \, dx = & \int_{\partial\Omega} \langle \nabla u, \nu \rangle \Big( N - \Delta_{\tau} u - (N-1)\langle \nabla u,\nu\rangle \mathcal{M} \Big) \,d\mathcal H^{N-1}
        \\
        &- \int_{\partial\Omega} \langle (\nabla \nu ) \nabla^\tau u, \nabla^\tau u \rangle \,d\mathcal H^{N-1} - \beta \int_{\partial\Omega} |\nabla^\tau u|^2 \,d\mathcal H^{N-1}  
        \\
        &-  \int_{\partial\Omega} \langle \nabla u , \nu \rangle \,d\mathcal H^{N-1}.
    \end{align*}
    By straightforward algebraic manipulations we get
    \begin{align*}
        \int_{\Omega} \Delta P \, dx = & - \int_{\partial\Omega} \langle (\nabla \nu ) \nabla^\tau u, \nabla^\tau u \rangle \,d\mathcal H^{N-1} - \beta \int_{\partial\Omega} |\nabla^\tau u|^2 \,d\mathcal H^{N-1} 
        \\
        &+ (N-1) \int_{\partial\Omega} \langle \nabla u , \nu \rangle \,d\mathcal H^{N-1}
        - (N-1) \int_{\partial\Omega} \langle \nabla u , \nu \rangle^2 \mathcal{M} \,d\mathcal H^{N-1} 
        \\
        &- \int_{\partial\Omega} \langle \nabla u , \nu \rangle \Delta_\tau u \,d\mathcal H^{N-1}.  
    \end{align*}
    Moreover, a simple integration by parts on $\partial\Omega$ (see Theorem \ref{thm: divergence part Om}) informs us that
    \begin{equation*}
        - \int_{\partial\Omega} \langle \nabla u , \nu \rangle \Delta_\tau u \,d\mathcal H^{N-1} = \int_{\partial\Omega} \beta u \, \Delta_\tau u \,d\mathcal H^{N-1} = - \beta \int_{\partial\Omega} |\nabla^\tau u|^2 \,d\mathcal H^{N-1}.
    \end{equation*}
    Finally, observing that
    \begin{align*}
        (N-1) \int_{\partial\Omega} \langle \nabla u , \nu \rangle \,d\mathcal H^{N-1}
        - (N-1) \int_{\partial\Omega} \langle \nabla u , \nu \rangle^2 \mathcal{M} \,d\mathcal H^{N-1} 
        \\
        = - (N-1) \int_{\partial\Omega} \langle \nabla u, \nu \rangle \Big( \langle \nabla u, \nu \rangle \mathcal{M} -1 \Big) \,d\mathcal H^{N-1},
    \end{align*}
    the conclusion follows.
\end{proof}

We are now in the position to state and prove an identity that generalizes to the Robin boundary conditions the identity of \cite[Theorem 2.2]{mp19}.

\begin{teo}[Identity for the Soap Bubble Theorem]\label{thm: identity for SBT}
    Let $\Omega$ be an open, connected, bounded subset of $\R^N$ with $C^2$ boundary. Let $\beta \in \mathbb{R} \setminus \{0\}$ with \eqref{beta condition} in force. Let $u\in C^2(\overline{\Omega})$ be a solution to the problem \eqref{PbR}. Let $P$ be as in \eqref{eq P}. Then the following identity holds:
    \begin{align*}
        \int_{\Omega} \Delta P \, dx +   (N-1)\int_{\partial\Omega} &\Big( \langle \nabla u ,\nu \rangle - R \Big)^2 \mathcal{M}_0 \,d\mathcal H^{N-1} 
        \\
        =& - \int_{\partial\Omega} \langle (\nabla \nu ) \nabla^\tau u, \nabla^\tau u \rangle \,d\mathcal H^{N-1} - 2 \beta \int_{\partial\Omega} |\nabla^\tau u|^2 \,d\mathcal H^{N-1} 
        \\
        & - (N-1) \int_{\partial\Omega} \langle \nabla u, \nu \rangle^2 \Big( \mathcal{M} -  \mathcal{M}_0 \Big) \,d\mathcal H^{N-1}, 
    \end{align*}
    where 
    \begin{equation*}
        R = \frac{N |\Omega|}{|\partial\Omega|} \quad\text{and}\quad \mathcal{M}_0 = \frac{1}{R}.
    \end{equation*}
    
\end{teo}

\begin{proof}
    We work on the identity of Theorem \ref{teo first identity}. We first observe that adding and subtracting the term 
    \[(N-1) \int_{\partial\Omega} \langle \nabla u, \nu \rangle^2 \mathcal{M}_0 \,d\mathcal H^{N-1},
    \]
    we deduce that 
    \begin{align*}
        - (N-1) \int_{\partial\Omega} &\langle \nabla u, \nu \rangle \Big( \langle \nabla u, \nu \rangle \mathcal{M} -1 \Big) \,d\mathcal H^{N-1} 
        \\
        = &- (N-1) \int_{\partial\Omega} \langle \nabla u, \nu \rangle^2 \Big( \mathcal{M} -  \mathcal{M}_0 \Big) \,d\mathcal H^{N-1}
        \\
        &- (N-1) \int_{\partial\Omega} \langle \nabla u, \nu \rangle^2 \mathcal{M}_0 \,d\mathcal H^{N-1} + (N-1) \int_{\partial\Omega} \langle \nabla u , \nu \rangle \,d\mathcal H^{N-1}.
    \end{align*}
    Then, by exploiting the simple relation
    \[
    \langle \nabla u, \nu \rangle^2 = \Big( \langle \nabla u ,\nu \rangle - R \Big)^2 + 2 \langle \nabla u, \nu \rangle R - R^2,
    \]
    we obtain 
    \begin{align*}
        (N-1) \int_{\partial\Omega} \langle \nabla u, \nu \rangle^2 \mathcal{M}_0 \,d\mathcal H^{N-1} & = (N-1) \int_{\partial\Omega} \Big( \langle \nabla u ,\nu \rangle - R \Big)^2 \mathcal{M}_0 \,d\mathcal H^{N-1} 
        \\
        &\quad + 2 (N-1) \int_{\partial\Omega} \langle \nabla u, \nu \rangle R \mathcal{M}_0 \,d\mathcal H^{N-1} 
        \\
        &\quad - (N-1) \int_{\partial\Omega} R^2 \mathcal{M}_0 \,d\mathcal H^{N-1}
        \\
        &= (N-1) \int_{\partial\Omega} \Big( \langle \nabla u ,\nu \rangle - R \Big)^2 \mathcal{M}_0 \,d\mathcal H^{N-1} 
        \\
        &\quad + (N-1) \int_{\partial\Omega} \langle \nabla u, \nu \rangle \,d\mathcal H^{N-1},
    \end{align*}
    where in the last equality we have used the definition of $R$ and $\mathcal{M}_0$ and the fact that, by Divergent Theorem, one has
    \begin{equation}\label{eq property of int u nu}
        \int_{\partial\Omega} \langle \nabla u, \nu \rangle \,d\mathcal H^{N-1}
        = \int_{\Omega} \Delta u \,dx = N |\Omega| = R |\partial\Omega| = \int_{\partial\Omega} R \,d\mathcal H^{N-1}.
    \end{equation}
    Summing up all previous formulas, we deduce that 
    \begin{align*}
        &- (N-1) \int_{\partial\Omega} \langle \nabla u, \nu \rangle \Big( \langle \nabla u, \nu \rangle \mathcal{M} -1 \Big) \,d\mathcal H^{N-1} 
        \\
        &= - (N-1) \int_{\partial\Omega} \langle \nabla u, \nu \rangle^2 \Big( \mathcal{M} -  \mathcal{M}_0 \Big) \,d\mathcal H^{N-1}
        -(N-1) \int_{\partial\Omega} \Big( \langle \nabla u ,\nu \rangle - R \Big)^2 \mathcal{M}_0 \,d\mathcal H^{N-1},
    \end{align*}
    and the proof is of the identity is complete. 
    
\end{proof}

We are now in the position to state and prove an identity when the overdetermined condition \eqref{spr intro} of the Introduction is in force.

\begin{teo}[Identity for the Serrin's problem with Robin boundary conditions] \label{thm: identity Serrin}
    Let $\Omega$ be an open, connected, bounded subset of $\R^N$ with $C^2$ boundary. Let $\beta \in \mathbb{R} \setminus \{0\}$ with \eqref{beta condition} in force. Let $u\in C^2(\overline{\Omega})$ be a solution to the problem \eqref{PbR}. Let $P$ be as in \eqref{eq P}.
    Assume that $u$ satisfies also
    \begin{equation*}
        |\nabla u|^2 + 2Nu - 2\beta^2 u^2 + (N-1) \beta u^2 \mathcal{M} = C \quad \text{on } \partial\Omega,
    \end{equation*}
    where 
    \begin{equation*}
        R = \frac{N |\Omega|}{|\partial\Omega|} \quad\text{and}\quad C = - R^2 - R \left(\frac{N+1}{\beta}\right).
    \end{equation*}
    Then the following identity holds:
    \begin{align*}
        \int_{\Omega} \Delta P \, dx + \beta \int_{\partial\Omega} \Big( \langle \nabla u ,\nu \rangle - R \Big)^2 \,d\mathcal H^{N-1} 
        = &- \int_{\partial\Omega} \langle (\nabla \nu ) \nabla^\tau u, \nabla^\tau u \rangle \,d\mathcal H^{N-1} 
        \\
        &- \beta \int_{\partial\Omega} |\nabla^\tau u|^2 \,d\mathcal H^{N-1}.
    \end{align*}
\end{teo}

\begin{proof}
    We multiply the overdetermined condition by $\beta$ and we integrate over $\partial\Omega$. We get
    \begin{align*}
        \beta \int_{\partial\Omega} |\nabla u|^2 \,d\mathcal H^{N-1}  +  2 N \int_{\partial\Omega} \beta u \,d\mathcal H^{N-1} &-  2 \beta \int_{\partial\Omega} \beta^2 u^2 \,d\mathcal H^{N-1}
        \\
        &+ (N-1) \int_{\partial\Omega} \beta^2 u^2 \mathcal{M} \,d\mathcal H^{N-1}
        = C \beta |\partial\Omega|.
    \end{align*}
    Hence, by the definition of $C$, we obtain that
    \begin{align*}
        (N-1) \int_{\partial\Omega} \langle \nabla u,\nu \rangle^2 \mathcal{M} \,d\mathcal H^{N-1}
        = &\left( - R^2 - R \left(\frac{N+1}{\beta}\right) \right) \beta |\partial\Omega|
        \\
        &-\beta \int_{\partial\Omega} |\nabla u|^2 \,d\mathcal H^{N-1}  -  2 N \int_{\partial\Omega} \beta u \,d\mathcal H^{N-1} 
        \\
        &+  2 \beta \int_{\partial\Omega} \langle \nabla u,\nu \rangle^2 \,d\mathcal H^{N-1}.
    \end{align*}
    We now substitute the above term into the identity provided by Theorem \ref{teo first identity}. We deduce that
    \begin{align*}
        \int_{\Omega} \Delta P \, dx = &- \int_{\partial\Omega} \langle (\nabla \nu ) \nabla^\tau u, \nabla^\tau u \rangle \,d\mathcal H^{N-1} - 2 \beta \int_{\partial\Omega} |\nabla^\tau u|^2 \,d\mathcal H^{N-1}
        \\
        &+ R^2 \beta |\partial\Omega| + R(N+1) |\partial\Omega| 
        +  \beta \int_{\partial\Omega} |\nabla u|^2 \,d\mathcal H^{N-1} + 2 N \int_{\partial\Omega} \beta u \,d\mathcal H^{N-1} 
        \\
        &-  2 \beta \int_{\partial\Omega} \langle \nabla u,\nu \rangle^2 \,d\mathcal H^{N-1} + (N-1) \int_{\partial\Omega} \langle \nabla u,\nu \rangle \,d\mathcal H^{N-1}
        \\
        = &- \int_{\partial\Omega} \langle (\nabla \nu ) \nabla^\tau u, \nabla^\tau u \rangle \,d\mathcal H^{N-1} - \beta \int_{\partial\Omega} |\nabla^\tau u|^2 \,d\mathcal H^{N-1} 
        \\
        &- \beta \int_{\partial\Omega} \langle \nabla u,\nu \rangle^2 \,d\mathcal H^{N-1}
        + R^2 \beta |\partial\Omega| 
        \\
        &+ R(N+1) |\partial\Omega| 
        - (N+1) \int_{\partial\Omega} \langle \nabla u,\nu \rangle \,d\mathcal H^{N-1}.
    \end{align*}
    To conclude we simply notice that, by \eqref{eq property of int u nu}, we have
    \[
     R(N+1) |\partial\Omega| 
        - (N+1) \int_{\partial\Omega} \langle \nabla u,\nu \rangle \,d\mathcal H^{N-1} = 0,
    \]
    while
    \[
    - \beta \int_{\partial\Omega} \langle \nabla u,\nu \rangle^2 \,d\mathcal H^{N-1}
        + R^2 \beta |\partial\Omega| = - \beta \int_{\partial\Omega} \Big( \langle \nabla u,\nu \rangle^2 -R^2 \Big) \,d\mathcal H^{N-1}.
    \]
    Hence, using the trivial relation
    \begin{equation*}
        \langle \nabla u, \nu \rangle^2 - R^2 = \Big( \langle \nabla u ,\nu \rangle - R \Big)^2 + 2 \langle \nabla u, \nu \rangle R - 2R^2,
    \end{equation*}
    we deduce that 
    \begin{align*}
        \int_{\Omega}& \Delta P \, dx = - \int_{\partial\Omega} \langle (\nabla \nu ) \nabla^\tau u, \nabla^\tau u \rangle \,d\mathcal H^{N-1} -  \beta \int_{\partial\Omega} |\nabla^\tau u|^2 \,d\mathcal H^{N-1} 
        \\
        & - \beta \int_{\partial\Omega} \Big( \langle \nabla u ,\nu \rangle - R \Big)^2 \,d\mathcal H^{N-1} - 2 \beta \int_{\partial\Omega} \langle \nabla u ,\nu \rangle R \,d\mathcal H^{N-1} + 2 \beta  \int_{\partial\Omega}  R^2 \,d\mathcal H^{N-1},
    \end{align*}
    and the identity of the statement follows again by an application of \eqref{eq property of int u nu} to the last two terms on the right hand-side.
\end{proof}

\section{Overdetermined problems}
\label{sec overdet problem}
In this section we are going to study some quoted overdetermined problems and prove the main results. As mentioned in the Introduction, by using the identity of Theorem \ref{thm: identity Serrin}, we get the following generalization of the celebrated Serrin's Theorem.

\begin{teo}[Rigidity for the Serrin's problem with Robin boundary conditions]\label{thm Serrin rigidity}
    Let $\Omega \subset \R^N$, $N\ge 2$, be an open, connected, bounded set with $C^2$ boundary. Let $\beta > 0$ and let $u \in C^2(\overline{\Omega})$ be a solution of
    \begin{equation}\label{SpR}
        \begin{cases}
        \Delta u = N &\text{in } \Omega,
        \\
        \displaystyle \frac{\partial u}{\partial \nu}+ \beta u = 0 &\text{on } \partial\Omega,
        \\
        |\nabla u|^2 + 2Nu - 2\beta^2 u^2 + (N-1) \beta u^2 \mathcal{M} = C &\text{on } \partial\Omega,
      \end{cases}
    \end{equation}
    with 
    \begin{equation*}
        R = \frac{N |\Omega|}{|\partial\Omega|} \quad\text{and}\quad C = - R^2 - R \left(\frac{N+1}{\beta}\right).
    \end{equation*}
    Let $ k_{\min}$ be as in \eqref{k}. If
    \begin{equation}
    \label{sc}
    \beta+k_{\min}\ge 0
    \end{equation}
     then $\Omega$ is a ball and $u$ is radially symmetric.
\end{teo}

\begin{proof}
    It suffices to apply the identity of Theorem \ref{thm: identity Serrin}. On one hand, it holds
    \begin{equation*}
        \Delta P=|\nabla^2 u |^2 - \frac{(\Delta u)^2}{N} ,
    \end{equation*}
    and then $\Delta P$ is non-negative by Cauchy-Schwarz inequality. 
    Since $\beta > 0$, we readily deduce that
    \begin{equation*}
        \int_{\Omega} \Delta P \, dx + \beta \int_{\partial\Omega} \Big( \langle \nabla u ,\nu \rangle - R \Big)^2 \,d\mathcal H^{N-1} \geq 0.
    \end{equation*}
    On the other hand, by \eqref{stima} and by the hypothesis \eqref{sc}, we get
    \begin{equation*}
    \begin{split}
        - \int_{\partial\Omega} \langle (\nabla \nu ) \nabla^\tau u, \nabla^\tau u \rangle \,d\mathcal H^{N-1} 
        &- \beta \int_{\partial\Omega} |\nabla^\tau u|^2 \,d\mathcal H^{N-1} 
        \\
        & \leq - (k_{\min} + \beta) \int_{\partial\Omega} |\nabla^\tau u|^2 \,d\mathcal H^{N-1}   
        \leq 0.
    \end{split}
    \end{equation*}
    Hence,
    \begin{equation*}
        \Delta P = 0 \quad\text{in }\Omega,
    \end{equation*}
    and
     \begin{equation}
     \label{ball}
        \langle \nabla u ,\nu \rangle=R \quad\text{ and }\quad  |\nabla^\tau u|=0 \quad  \text{ on }\partial\Omega.
    \end{equation}
   Being $\Delta P=0$ and $\Delta u=N$, then $u$ is a quadratic polynomial of the form 
    \begin{equation*}
        u(x) = \frac{1}{2} \left( |x-z|^2 + a \right) \quad \text{for all } x \in \overline{\Omega}, 
    \end{equation*}
    for suitable fixed $z \in \Omega$ and $a \in \R$.
    Moreover by \eqref{ball} and by the Robin boundary conditions, $u$ is constant on the boundary then $\Omega=B_R(z) $  and $a = -\frac{2R}{\beta} - R^2$ (cf. \eqref{eq q radial}).
   The proof is complete.
\end{proof}
\begin{rem}
We stress that the previous rigidity result holds in particular for convex sets and for any $\beta >0$.
Moreover we notice that, for $\beta \to +\infty$, condition \eqref{sc} is trivial and   the problem \eqref{SpR} converges to the  Serrin problem
    \begin{equation*}
        \begin{cases}
        \Delta u = N &\text{in } \Omega,
        \\
        u = 0 &\text{on } \partial\Omega,\\
        \displaystyle \frac{\partial u}{\partial \nu} = R &\text{on } \partial\Omega.
      \end{cases}
    \end{equation*}
    Then our result recovers the one known in the Dirichlet case.
\end{rem}

Then, the identity of Theorem \ref{thm: identity for SBT} provides an alternative proof of the celebrated Alexandrov's Theorem using a suitable Robin problem.

\begin{teo}[Soap Bubble Theorem]
    Let $\Omega$ be an open, connected, bounded subset of $\R^N$ with $C^2$ boundary. Let
    \begin{equation*}
        R = \frac{N |\Omega|}{|\partial\Omega|} \quad\text{and}\quad \mathcal{M}_0 = \frac{1}{R}.
    \end{equation*}
    If the mean curvature $\mathcal{M}$ of $\partial \Omega$ satisfies the inequality 
    \begin{equation*}
        \mathcal{M} \geq \mathcal{M}_0 \quad \text{on } \partial\Omega,
    \end{equation*}
    then $\partial\Omega$ must be a sphere (and hence $\Omega$ is a ball) of radius $R$.
    In particular, the same conclusion holds if $\mathcal M$ equals some constant on $\partial\Omega$.
\end{teo}
\begin{proof}
    The claim follows by using the identity provided by Theorem \ref{thm: identity for SBT}. Indeed, we can fix a parameter $\beta>0$ such that $\kappa_{\min} + 2\beta \geq 0$ and consider the solution  $u \in C^2(\bar{\Omega})$  to the problem \eqref{PbR}. Then, from the identity of Theorem \ref{thm: identity for SBT} and being $\Delta P \geq 0$ in $\Omega$, one readily deduces that 
    \begin{equation*}
        \int_{\Omega} \Delta P \, dx +   (N-1)\int_{\partial\Omega} \Big( \langle \nabla u ,\nu \rangle - R \Big)^2 \mathcal{M}_0 \,d\mathcal H^{N-1} \geq 0.
    \end{equation*}
    On the other hand, by the hypothesis $\mathcal{M} \geq \mathcal{M}_0$ on $\partial\Omega$ and by the assumption on $\beta$, we deduce that 
    \begin{align*}
        - \int_{\partial\Omega} \langle (\nabla \nu ) \nabla^\tau u, \nabla^\tau u \rangle \,d\mathcal H^{N-1} &- 2 \beta \int_{\partial\Omega} |\nabla^\tau u|^2 \,d\mathcal H^{N-1} 
        \\
        & - (N-1) \int_{\partial\Omega} \langle \nabla u, \nu \rangle^2 \Big( \mathcal{M} -  \mathcal{M}_0 \Big) \,d\mathcal H^{N-1} \leq 0,
    \end{align*}
    and the conclusion follows as in the last part of the proof of Theorem \ref{thm Serrin rigidity}.

    Finally, if $\mathcal{M}$ equals some constant on $\partial\Omega$, then the Minkowski's identity
    \begin{equation*}
        \int_{\partial\Omega} \mathcal{M} \langle x-z, \nu \rangle \,d\mathcal H^{N-1} = |\partial\Omega| 
    \end{equation*}
    implies that $\mathcal{M} = \mathcal{M}_0$ and we conclude.
    
\end{proof}

\begin{rem}
    We notice that the same theorem was proven in \cite[Theorem 2.2]{mp19} using the Dirichlet torsion problem, i.e. the solution to the following
    \begin{equation*}
        \begin{cases}
        \Delta u = N &\text{in } \Omega,
        \\
        u = 0 &\text{on } \partial\Omega.
      \end{cases}
    \end{equation*}
Here we obtain the same result using a suitable Robin torsion problem.
\end{rem}

Finally a result in the spirit of \cite[Theorem 2.4]{mp19}.

\begin{teo}
    Let $\Omega$ be an open, connected, bounded subset of $\R^N$ with $C^2$ boundary. Let $\beta \in \mathbb{R}\setminus \{0\}$ with \eqref{beta condition} in force. Let $u\in C^2(\overline{\Omega})$ be a solution to the following overdetermined problem  
    \begin{equation}\label{op_curv 1}
    \begin{cases}
    \Delta u = N &\text{in } \Omega,
    \\
    \dfrac{\partial u}{\partial \nu}+ \beta u = 0 &\text{on } \partial\Omega,
    \\
    \dfrac{\partial u}{\partial \nu} \mathcal{M} = 1 &\text{on } \partial\Omega.
    \end{cases}
\end{equation}
If $\displaystyle \kappa_{\min} + 2\beta > 0$, then $\Omega$ is a ball and $u$ is radially symmetric.
\end{teo}

\begin{proof}
    The argument is the same as before, simply using the identity of Theorem \ref{teo first identity}. It suffices to notice that, being $\langle \nabla u, \nu \rangle \mathcal{M} = 1$ on $\partial \Omega$ and $\displaystyle \kappa_{\min} + 2\beta > 0$, one obtains
    \begin{align*}
        0 \leq \int_{\Omega} \Delta P \, dx &= - \int_{\partial\Omega} \langle (\nabla \nu ) \nabla^\tau u, \nabla^\tau u \rangle \,d\mathcal H^{N-1} - 2 \beta \int_{\partial\Omega} |\nabla^\tau u|^2 \,d\mathcal H^{N-1} 
        \\
        & \quad - (N-1) \int_{\partial\Omega} \langle \nabla u, \nu \rangle \Big( \langle \nabla u, \nu \rangle \mathcal{M} - 1 \Big) \,d\mathcal H^{N-1}
        \\
        &
        \leq - (k_{\min} + 2\beta) \int_{\partial\Omega} |\nabla^\tau u|^2 \,d\mathcal H^{N-1}   
        \leq 0,
    \end{align*}
    and the conclusion follows again as in the last part of the proof of Theorem \ref{thm Serrin rigidity}. 
    
    We mention that, from $\nabla^\tau u = 0$ and by the Robin boundary conditions, the possibility that $\langle \nabla u, \nu \rangle$ equals a constant on $\partial\Omega$ different from $R$ is ruled out by the Divergence Theorem, being $\Delta u = N$ in $\Omega$.
\end{proof}

We conclude with a rigidity result involving an overdetermined problem, where the extra condition concerns the second normal derivative of the solution on the boundary. As mentioned in the Introduction, in contrast with the Dirichlet case, the overdetermined condition in problem \eqref{op_curv 2} below is not clear to be equivalent to the overdetermined condition in problem \eqref{op_curv 1}, in the Robin setting.

\begin{teo}
    Let $\Omega$ be an open, connected, bounded subset of $\R^N$ with $C^2$ boundary. Let $\beta \in \mathbb{R}\setminus \{0\}$ with \eqref{beta condition} in force. Let $u\in C^2(\overline{\Omega})$ be a solution to the following overdetermined problem 
    \begin{equation}\label{op_curv 2}
    \begin{cases}
    \Delta u = N &\text{in } \Omega,
    \\
    \dfrac{\partial u}{\partial \nu}+ \beta u = 0 &\text{on } \partial\Omega,
    \\
    \displaystyle \langle \nabla^2 u \, \nu , \nu \rangle = 1 &\text{on } \partial\Omega,
    \end{cases}
\end{equation}
If $\displaystyle \kappa_{\min} + \beta > 0$, then $\Omega$ is a ball and $u$ is radially symmetric. 
    
\end{teo}

\begin{proof}
    Simply following the proof of Theorem \ref{teo first identity} along with the Remark \ref{remark curvature}, one readily obtains that
    \begin{align*}
        0 \leq \int_{\Omega} \Delta P \, dx &= - \int_{\partial\Omega} \langle (\nabla \nu ) \nabla^\tau u, \nabla^\tau u \rangle \,d\mathcal H^{N-1} - \beta \int_{\partial\Omega} |\nabla^\tau u|^2 \,d\mathcal H^{N-1} 
        \\
        & \quad - \int_{\partial\Omega} \langle \nabla u, \nu \rangle \Big( 1 - \langle \nabla^2 u \, \nu , \nu \rangle \Big) \,d\mathcal H^{N-1} 
        \\
        &
        \leq - (k_{\min} + \beta) \int_{\partial\Omega} |\nabla^\tau u|^2 \,d\mathcal H^{N-1}   
        \leq 0, 
    \end{align*}
    where we have used the hypotheses $ \displaystyle \langle \nabla^2 u \, \nu , \nu \rangle = 1$ on $\partial \Omega$ and $\displaystyle \kappa_{\min} + \beta > 0$. Then, one concludes as in the last part of the proof of Theorem \ref{thm Serrin rigidity}.
\end{proof}

\subsection*{Acknowledgments}
This work has been partially supported by GNAMPA of INdAM. 

\noindent
Nunzia Gavitone was supported by the Project MUR PRIN-PNRR 2022: ``Linear and Nonlinear PDE'S: New directions and Applications”, P2022YFA. 

\noindent
Riccardo Molinarolo is supported by the Project ``Giochi a campo medio, trasporto e ottimizzazione in sistemi auto-organizzati e machine learning”, funded by MUR, D.D. 47/2025, PNRR - Missione 4, Componente 2, Investimento 1.2 - funded by European Union NextGenerationEU, CUP B33C25000380001.

\end{document}